\documentclass{amsart}
\usepackage{amsmath, amsxtra}
\usepackage{verbatim}

\usepackage[margin=1.52in]{geometry}
\usepackage{amssymb}

\usepackage{graphicx}

\usepackage[cmtip,all]{xy}

\usepackage{xypic}
\usepackage[colorlinks, linkcolor=blue]{hyperref}

\theoremstyle{plain}
\newtheorem{theorem}{Theorem}

\newtheorem{prop}{Proposition}[section]

\newtheorem{lemma}{Lemma}[section]

\theoremstyle{definition}

\newtheorem{definition}{Definition}

\newtheorem{remark}{Remark}

\def\ra{\rightarrow}
\def\rest{\text{\rm rest}}

\newcommand{\gitq}{/\hspace{-0.25pc}/}

\def\co{\colon\thinspace} 

\DeclareMathOperator{\Pic}{Pic}

\DeclareMathOperator{\proj}{Proj}

\DeclareMathOperator{\mult}{mult}

\DeclareMathOperator{\Sing}{Sing}
\DeclareMathOperator{\Exc}{Exc}

\DeclareMathOperator{\NS}{NS}

\def\Hn1{\mathcal{H}_{n,1}}

\renewcommand\min{\text{min}}

\def\A{\mathcal{A}}
\def\B{\mathcal{B}}
\def\C{\mathcal{C}}

\def\O{\mathcal{O}}

\newcommand\Mg[1]{\overline{\mathcal{M}}_{#1}}

\def\QQ{\mathbb{Q}}
\def\ZZ{\mathbb{Z}}

\def\PP{\mathbb{P}}
\def\ZZ{\mathbb{Z}}

\DeclareMathOperator\SL{SL}

\def\nb{\nobreakdash}

\begin{document}
\title{Moduli of weighted stable curves and log canonical models of $\Mg{g,n}$}
\author{Maksym Fedorchuk}
\address{Department of Mathematics, Columbia University, 2990 Broadway, New York, NY 10027}
\curraddr{}
\email{mfedorch@math.columbia.edu}

\begin{abstract} We prove that Hassett's spaces $\Mg{g,\A}$ are log canonical models of $\Mg{g,n}$.
\end{abstract}
\maketitle

\section{Introduction}
\label{S:introduction}

In \cite{Hassett-weighted}, Hassett constructed spaces $\Mg{g,\A}$ of {\em weighted pointed
stable curves}, each of which is a compactification of  $\mathcal{M}_{g,n}$. 
We establish nefness (and in some cases ampleness) of a number of divisors on $\Mg{g,\A}$ by exploiting the 
polarizing line bundle on the universal family over $\Mg{g,\A}$, which is known to be nef by Koll\'{a}r's semipositivity results \cite{kollar}. 
 As an application, we prove that the coarse moduli space of $\Mg{g,\A}$ is a log canonical model of $\Mg{g,n}$ (Theorem \ref{T:hassett-spaces}) for an arbitrary weight vector $\A$. 
In the case of $g=0$, this solves Problem 7.1 of \cite{Hassett-weighted} (see also ``Concluding Remarks'' of \cite{Alexeev-Swinarski}) and extends the results of \cite{Alexeev-Swinarski}, \cite{FS} and \cite{Simpson}.

We now recall the definition of $\Mg{g,\A}$ from \cite{Hassett-weighted}: Given a {\em weight vector} $\A=(a_1,\dots, a_n)\in (0,1]^n\cap \QQ^n$, we say that an $n$\nb-pointed
connected at-worst nodal curve $(C; p_1,\dots, p_n)$ of genus $g$ is {\em $\A$\nb-stable} if 
\begin{enumerate}
\item $p_1,\dots, p_n$ are nonsingular points of $C$.
\item $\omega_C(\sum_{i=1}^n a_ip_i)$ is ample.
\item $\mult_{x} \sum_{i=1}^n a_ip_i\leq 1$, $\forall x\in C$.
\end{enumerate}
The objects of the moduli stack $\Mg{g,\A}$ are flat families whose geometric fibers are $\A$\nb-stable curves of genus $g$.

Consider the universal family $(\pi\co \C \ra \Mg{g,\A}, \{\sigma_i\}_{i=1}^n)$ of $\A$\nb-stable curves;
here, $\{\sigma_i\}_{i=1}^n$ are the universal sections. By the definition above, the line bundle 
$\omega_{\pi}(\sum_{i=1}^n a_i\sigma_i)$ is $\pi$\nb-nef. It turns out that it is also nef on $\C$. We use this fact to produce nef divisors on $\Mg{g,\A}$. That $\omega_{\pi}(\sum_{i=1}^n a_i\sigma_i)$ is nef on $\C$ was proved by Koll\'{a}r
in \cite[Corollary 4.6 and Proposition 4.7]{kollar}. When $a_i=1$, $i=1,\dots,n$, i.e.  $\Mg{g,\A}=\Mg{g,n}$,
the nefness of $\omega_{\pi}(\sum_{i=1}^n \sigma_i)$ was also established by Keel \cite[Theorem 0.4]{keel}. We give yet another proof that $\omega_{\pi}(\sum_{i=1}^n a_i\sigma_i)$
is nef in Proposition \ref{P:relatively-nef}; the argument is elementary and is included because it gives a finer information of curves on which $\omega_{\pi}(\sum_{i=1}^n a_i\sigma_i)$ has degree $0$. 
The latter is used to establish ampleness of certain log canonical divisors on 
$\Mg{g,\A}$ (Theorem \ref {T:positive-divisors}). 

Even though we phrase all the results for
an arbitrary genus, the positive genus case follows from that of genus $0$ using 
a well-known positivity result on $\Mg{g}$. For this reason, we isolate nef
divisors on $\Mg{0,n}$ produced by our construction in Section \ref{S:F-cone}. These include all divisors contained in the {\em GIT 
cone} of $\Mg{0,n}$ \cite{Alexeev-Swinarski}, as shown in Proposition \ref{P:GIT}. We leave open the question whether 
two cones are in fact the same.

We work over an algebraically closed base field of an arbitrary characteristic.

\section{Main positivity result} 
\label{S:main}

Let $B$ be a smooth curve. Consider a flat proper family $\pi\co \C  \ra B$ of connected at-worst nodal curves of arithmetic genus $g$.
Suppose that $\pi$ has $n$, not necessarily distinct, sections 
\[
\sigma_1,\dots,\sigma_n\co B \ra \C.
\]
 Denote $\sigma_i(B)$ by $\Sigma_i$, and 
$\sum_{i=1}^n \Sigma_i$ by $\Sigma$. From now
on, we assume that every $\Sigma_i$ lies in $\C \setminus \Sing(\C/B)$. We denote both the relative dualizing line bundle of $\pi$ and its 
numerical class by
$\omega$. We denote the numerical class of a fiber of $\pi$ by $F$. 

\begin{prop}\label{P:relatively-nef} Let $\C \ra B$ be a generically smooth family of nodal curves of arithmetic genus $g$. Consider a weight vector $\A=\bigl(a_1,\dots, a_n\bigr)\subset (0,1]^n\cap \QQ^n$ such that 
\[
L:=\omega+\sum_{i=1}^n a_i\Sigma_i
\]
is $\pi$-nef. Suppose further that sections $\Sigma_{i_1}, \dots, \Sigma_{i_k}$ 
can coincide only if $\sum_{\ell=1}^k a_{i_\ell}\leq 1$.
Then $L$ is a nef divisor on $\C$.  
 \end{prop}

\begin{remark} The assumptions of the above proposition are weaker than the condition 
of $\A$-stability (see Section \ref{S:introduction}) in that  
they allow a collection of sections to intersect at finitely many points even 
when the sum of the associated weights is greater than $1$.
\end{remark} 
\begin{proof} 
The surface $\C$ has at-worst $A_k$ singularities, which are Du Val (the discrepancies of the canonical divisor are zero). It follows 
that the statement holds for $\C$ if and only if it holds for the minimal desingularization of $\C$. From now on, we assume that 
$\C$ is smooth. 

\begin{proof}[Proof in the case $g\geq 2$] Assume that $C\subset \C$ is an irreducible curve not contained in the fiber of $\pi$. 
Denote by $f\co \C \ra \C_\min$ the relative minimal model of $\C$ over $B$. We have $\omega_{\C/B}=f^*\omega_{\C_\min/B}+E$, 
where $E$ is an effective vertical divisor. By \cite[Theorem 6.33]{HM}, $\omega_{\C_\min/B}$ is nef on $\C_\min$. It follows that 
$\omega_{\C/B}\cdot C\geq 0$ and $-\Sigma_i^2=\omega_{\C/B}\cdot \Sigma_i\geq 0$. We
conclude that $L\cdot C\geq 0$. 
Together with $\pi$-nefness assumption this proves that $L$ is nef.
\end{proof}
\begin{proof}[Proof in the case $g=1$.] 
In this case, $\omega_{\C/B}$ is a linear combination of effective divisors supported in fibers (see \cite[Theorem 15, p. 176]{friedman}). Therefore, $-\Sigma_i^2=\omega_{\C/B}\cdot \Sigma_i\geq 0$. 
By the virtue of effectivity of $\omega_{\C/B}$ and $\pi$-nefness assumption, to prove that $L=\omega_{C/B}+\sum_{i=1}^n a_i\Sigma_i$ is nef, we need to show that 
$L\cdot \Sigma_i\geq 0$, for every $i=1,\dots, n$. 
Let $\Sigma_{j_1},\dots, \Sigma_{j_k}$ be all the sections coinciding with $\Sigma_i$. Then
\begin{align*}
L\cdot \Sigma_i &\geq  (\omega_{\C/B}+\sum_{\ell=1}^k a_{j_\ell}\Sigma_{j_\ell})\cdot \Sigma_i \\
&= -\Sigma^2_i (1-\sum_{\ell=1}^k a_{j_\ell}) \geq 0.
\end{align*}
\renewcommand\qedsymbol{\empty}
\end{proof}
\begin{proof}[Proof in the case $g=0$.] 

We need to show that $C\cdot L\geq 0$ for every irreducible curve $C$ on $\C$. When $C$ lies in the fiber, $C\cdot L\geq 0$
by assumption. Next, the proof breaks into two parts: in the first, we deal with the case when $C$ is a section; in the second, we show
that a general case reduces to the former.

First, suppose that $C$ is a section. By successively blowing-down $(-1)$\nb-curves not meeting $C$, we reduce to the 
case when $\C\ra B$ is a $\PP^1$\nb-bundle. We consider two cases: $C^2<0$ and $C^2\geq 0$.

{\em Case \rm{I}: $C$ is a negative section of $\C\ra B$.} Without loss of generality, we can assume that sections 
$\Sigma_1,\dots,\Sigma_k$
coincide with $C$. Then compute 
\begin{align*}
C\cdot L&=C\cdot (\omega+\sum_{i=1}^n a_i\Sigma_i) \\
&=-C^2+\sum_{i=1}^k a_iC^2+\sum_{i=k+1}^n C\cdot a_i\Sigma_i \\
&=-C^2(1-\sum_{i=1}^k a_i)+\sum_{i=k+1}^n a_i(C\cdot \Sigma_i) \geq 0,
\end{align*}
since $\sum_{i=1}^k a_i\leq 1$ by assumption.

{\em Case \rm{II}: $C$ is a non-negative sections of $\C\ra B$.} To make the computation more transparent, let 
$\{\Sigma_i\}_{i=1}^k$, $\{\Sigma_i\}_{i=k+1}^\ell$ and $\{\Sigma_i\}_{i=\ell+1}^n$ be the sections 
that, respectively, coincide with $C$, have negative self-intersection and are neither of the first two. 
A $\PP^1$-bundle $\C\ra B$
has a unique section of negative self-intersection, say, $E^2=-r$. Therefore, $\Sigma_i=E$  for all $i=k+1,\dots,\ell$.
By our assumption, we have
\begin{align*}
\sum_{i=1}^k a_i &\leq 1, \\
\sum_{i=k+1}^\ell a_i &\leq 1, \\
\sum_{i=1}^n a_i &\geq 2.
\end{align*}
A non-negative section $\Sigma$ on $\C$ satisfies $\Sigma^{2}\geq r$. 
Also, since $(C-\Sigma_i)^2=0$, we have 
\[
C\cdot \Sigma_i=\frac{1}{2}(C^2+\Sigma_i^2), \quad i=1,\dots, n.
\]
Now compute
\begin{align*}
C\cdot L&=C\cdot (\omega+\sum_{i=1}^n a_i\Sigma_i) \\
&=-C^2+\sum_{i=1}^n (C\cdot a_i\Sigma_i) \\
&=C^2(\sum_{i=1}^k a_i+\sum_{i=k+1}^n \frac{a_i}{2} -1)+\frac{1}{2}(-r\sum_{i=k+1}^\ell a_i+\sum_{i=\ell+1}^n a_i\Sigma^2_i) \\
&\geq r(\sum_{i=1}^k a_i +\sum_{i=\ell+1}^n a_i-1) >0.
\end{align*}
 
Finally, we consider a general case when $C$ is an arbitrary irreducible curve on $\C$ that does not lie in the fiber of $\pi$. 


Suppose first that $C^2< 0$. The intersection number $\omega\cdot C$ equals to $-C^{2}$ plus the number 
of branch points of $\tilde{C}\ra B$, where $\tilde{C}$ is the normalization of $C$. Hence $\omega\cdot C>0$. If $C\neq \Sigma_i$, then clearly $C\cdot L>0$.
The case $C=\Sigma_i$ is already considered.

Suppose now that $C^2\geq 0$. We will proceed to prove by contradiction that $C\cdot L\geq 0$.
If $C\cdot L<0$, then for an arbitrary ample divisor $H$ on $\C$, we can find 
$0<\epsilon\ll 1$ such that $(C+\epsilon H)\cdot L<0$. Since $C$ is a nef divisor, $C+\epsilon H$ is ample by the
Kleiman's criterion \cite{kl}. In particular, we can find a smooth curve $D\subset \C$ whose numerical class is a multiple of $C+\epsilon H$ 
and such that $D$ avoids (finitely many points) $\Sing(\C/B)$. Consider the induced map $D\ra B$ and the fiber product
$\C':=\C\times_B D$. 
\begin{align*}
\xymatrix{\C'=\C\times_B D \ar[d] \ar[r] & \C\ar[d]^{\pi}  \\
D\ar[r] \ar@_{^{`}->}[ur] \ar@/^1pc/[u]^{\tau} & B 
}
\end{align*} 
The projection $\C'\ra D$ has a section $\tau\co D \ra \C'$, whose image in $\C'$ 
maps onto $D$ in $\C$. Let $\Sigma_i'$ be the preimage of $\Sigma_i$ on $\C$. Introduce the line bundle
\[
L'=\omega_{\C'/D}+\sum_{i=1}^n a_i\Sigma'_i.
\]
Then $L'$ is the pullback of $L$ from $\C$. Moreover, $\tau(D)\cdot L'=D\cdot L$, by the projection formula. 
Noting that  the surface $\C'$ and the line bundle $L'$ satisfy the assumptions of the proposition, we must have $\tau(D)\cdot L'\geq 0$, by the first part of the argument. A contradiction.
\end{proof}
\renewcommand{\qedsymbol}{\empty}
\end{proof}

\section{Nef divisors on $\Mg{g,\A}$}
\subsection{Tautological divisors}
\label{S:divisors}
Let $\A=\bigl(a_1,\dots, a_n\bigr)$ be a weight vector with $a_i\in (0,1]\cap \QQ$ and
$2g-2+\sum\limits_{i=1}^n a_i>2$. 
Consider the Hassett's moduli space $\Mg{g,\A}$ parameterizing $\A$\nb-stable pointed genus $g$ curves \cite{Hassett-weighted}. 
It is a smooth Deligne-Mumford stack and carries a universal family 
$(\pi\co \C\ra \Mg{g,\A}, \sigma_1,\dots, \sigma_n)$. We consider the following 
{\em tautological divisor classes} on the stack $\Mg{g,\A}$:
\begin{enumerate}
\item 
The {\em Hodge} class $\lambda=c_1(\pi_*\omega_{\C/B})$.
\item 
The {\em kappa} class $\kappa=\pi_*(c^2_1(\omega_{C/B}))$ (this is different from $\kappa_1$ of \cite{AC}).
\item 
The {\em psi}-classes $\psi_i=\pi_*(-\sigma_i^2)$; the total {\em psi}-class is $\psi:=\sum\limits_{i=1}^n \psi_i$.
\item The boundary divisors $\Delta_{i,j}:=\pi_*(\sigma_i\cdot \sigma_j)$. These are denoted $D_{I,J}(\A)$, where 
$I=\{i,j\}$ and $J=\{1,\dots, n\}\setminus I$, in \cite{Hassett-weighted}. 
The sum of all $\Delta_{i,j}$ is denoted by $\Delta_s$.
\item The sum, denoted $\Delta_{nodal}$, of the boundary divisors parameterizing nodal curves. This is denoted by $\nu$ in {\em loc.cit.}
\end{enumerate}
The Mumford's relation gives $\kappa=12\lambda-\Delta_{nodal}$ (cf. \cite{AC}). In the 
case of $\Mg{g,n}$, we use the usual notation $\Delta$ for the total boundary. Note for the future that under the natural reduction morphism $\Mg{g,n} \ra \Mg{g,\A}$, the divisor $\Delta$ in $\Mg{g,n}$ pushes forward to $\Delta_s+\Delta_{nodal}$. 

In the case of $\Mg{g,n}$, we also consider boundary divisors $\Delta_S$, 
for every subset $S\subset \{1,\dots, n\}$ satisfying
$|S|\geq 2$ (and in the case of $g=0$, $n-|S|\geq 2$). The generic point of $\Delta_S$ is a reducible curve with a single node and an irreducible component of genus $0$ marked
by sections $\{\sigma_i\}_{i\in S}$.

\subsection{Main result} 
\begin{theorem}\label{T:positive-divisors}
The following divisors 
\begin{align*}
A &=A(a_{1},\dots,a_{n})= \kappa+\psi+\sum_{i<j}(a_i+a_j)\Delta_{ij},  \\
B &=B(a_{1},\dots,a_{n}) = \kappa+\sum_{i=1}^n (2a_i-a_i^2)\psi_i+\sum_{i<j}(2a_ia_j) \Delta_{ij},\\
C_i &=C_i(a_{1},\dots,a_{n}) = (1-a_{i})\psi_{i}+\sum_{j\neq i}a_j\Delta_{ij}, \quad \text{for each $i=1,\dots, n$,}
\end{align*}
are nef on $\Mg{g,\A}$.
 \end{theorem}
 \begin{proof}
Adopt the convention that for a generically singular family $\C\ra B$, 
the sections $B\ra \C$ associated to the conductor have weight $1$. 
Then the divisors under consideration are functorial with respect to the boundary stratification (cf. \cite{FS}). Therefore, it suffices to check that they
have non-negative degree on every curve $\pi\co (\C; \sigma_1,\dots, \sigma_n)\ra B$ with an irreducible total space $\C$.

 First, by Proposition \ref{P:relatively-nef}, the divisor 
 \[
 L:=\omega+\sum_{i=1}^n a_i\sigma_i
 \]
 is nef on $\C$. In particular, it is pseudoeffective and has a non-negative 
 self-intersection. We will show that the intersection 
 numbers 
 $A$, $B$ and $C_i$ are non-negative by expressing each of them as an intersection of $L$ 
 with a pseudoeffective class on $\C$.
 
 For $A$, we note that $\omega+\sum_{i=1}^n\sigma_i$ is an effective combination of $L$ and $\sigma_i$, $1\leq i\leq n$. Therefore,
\[
0\leq (\omega+\sum_{i=1}^n \sigma_i)\cdot L=\kappa+\psi+\sum_{i<j}(a_i+a_j)\Delta_{ij}=A.
\]
 
 For $B$, we have
 \[
 0\leq L^2=(\omega+\sum_{i=1}^n a_i\sigma_i)^2=\kappa+\sum_{i=1}^n (2a_i-a_i^2)\psi_i+\sum_{i<j}(2a_ia_j) \Delta_{ij}=B. 
 \]

For $C_i$, we have
\[
0\leq L\cdot \sigma_{i}=(1-a_{i})\psi_{i}+\sum_{j\neq i}a_j\Delta_{ij}=C_i.
\]
\end{proof}

\subsection{Ampleness result}
With a bit more work, we can prove that the divisor $A$ is ample. We begin with a few preliminaries.
\begin{definition}
\label{D:replacement}
 Let $\A=(a_1,\dots, a_n)$ be a weight vector. Suppose that $a_1=\sum_{\ell=1}^k b_{\ell}$.
Set $\B=(b_1,\dots,b_{k}, b_{k+1},\dots, b_{n+k-1})$, where $b_{k+i}=a_{i+1}$ for $i\geq 1$. 
\par
A {\em replacement morphism} $\chi\co \Mg{g,\A} \ra \Mg{g,\B}$ is a morphism which sends an $\A$-stable curve $(\C\ra B; \sigma_1,\dots,
\sigma_n)$ to a $\B$-stable curve $(\C\ra B; \tau_{1},\dots, \tau_{k}, \sigma_2,\dots, \sigma_{n})$, where the sections 
$\{\tau_i\}_{1\leq i \leq k}$ are equal to the section $\sigma_1$.
\end{definition}

\begin{lemma}[Pull-back formulae for $\chi$]\label{L:replacement} Keep the notation of Definition \ref{D:replacement}.
Under the replacement morphism $\chi:\Mg{g,\A} \ra \Mg{g,\B}$ the divisor $A(b_1,\dots, b_{n+k-1})$ pulls back according to the following formula
\begin{align}
 \chi^*\left(A(b_1,\dots, b_{n+k-1})\right)=A(a_1,\dots, a_n)+(1-a_1)\psi_1.
\end{align}
\end{lemma}
\begin{proof} This follows from a straightforward generalization of \cite[Lemma 2.9]{FS}. The proof is analogous and so is omitted.
\end{proof}

\begin{lemma}
\label{L:effectivity} The divisor $\kappa+\psi$ is a positive linear combination of $\Delta_s$ and all boundary divisors on the space of weighted
pointed rational curves $\Mg{0,\A}$.
\end{lemma}
\begin{proof}
On $\Mg{0,n}$, we have $\kappa=-\Delta$ and a well-known relation 
\begin{align}
\label{E:positivity}
\kappa+\psi=\sum_{r=2}^{\lfloor n/2 \rfloor} \frac{r(n-r)-n+1}{n-1}\Delta_r.
\end{align}
This is a positive combination of $\Delta_2$ and all other boundary divisors for $n\geq 4$. The general case follows by pushing forward the relation 
\eqref{E:positivity} to $\Mg{0,\A}$ via the reduction morphism.
\end{proof}

\begin{prop}
\label{P:ample-divisors} 
The divisor $A=A(a_{1},\dots,a_{n})=\kappa+\psi+\sum_{i<j}(a_i+a_j)\Delta_{ij}$  on $\Mg{g,\A}$ is a pullback of an ample divisor on the coarse moduli space.
\end{prop}
\begin{proof}
We claim that a sufficiently small neighborhood (in the Euclidean topology with the usual norm $\Vert \cdot \Vert$) of $A$  lies inside
the nef cone of the vector space $\NS(\Mg{g,\A})\otimes_{\ZZ}\QQ$. The statement then follows from Kleiman's criterion \cite{kl} applied to the coarse moduli space of
$\Mg{g,\A}$.
The proof is by induction on dimension. When dimension is $1$ (i.e., when $g=0$ and $|\A|=4$, or $g=1$ and $|\A|=1$) the proof is by direct computation.
\par
Suppose $\dim \Mg{g,\A} > 1$. By the functoriallity of $A$ 
and by the induction assumption, 
a sufficiently small perturbation of $A$ is still ample when restricted to any boundary divisor
in $\Delta_{n}$.  
It remains to show that for any $P\in \NS(\Mg{g,\A})\otimes\QQ$ satisfying $0\leq \Vert P \Vert \ll 1$, the perturbed divisor $A+P$ has non-negative
degree 
on any family $\C\ra B$ with a generically smooth fiber of genus $g$.
\par
The proof falls into two parts: $g=0$ and $g\geq 1$.
\par
{\em Case of $g=0$}:
First, suppose that no sections of negative self-intersection 
are coincident. 
Then for $0< \epsilon \ll 1$  
the divisor $(1-\epsilon)\omega+\sum a_i \sigma_i$ is nef on $\C$ by Proposition \ref{P:relatively-nef} (see also Remark 1).
It follows that the divisor 
\begin{align*}
A(\epsilon) 
&=\left((1-\epsilon)\omega+\sum a_i\sigma_i \right)
(\omega+\sum \sigma_i) =A-\epsilon(\kappa+\psi).
\end{align*}
is non-negative on $B$. By Lemma \ref{L:effectivity}, the divisor $\kappa+\psi$ is a positive linear combination of 
all boundary divisors and all divisors in $\Delta_s$. 
Since irreducible
components of these 
divisors generate $\NS(\Mg{0,\A})$, we conclude that 
$\epsilon(\kappa+\psi)+P$ 
is an effective linear combination of these generators. 
Therefore,
$A+P=A(\epsilon)+\epsilon(\kappa+\psi)+P$ intersects $B$ non-negatively.

Suppose that $k\geq 2$ sections $\{\sigma_{i_\ell}\}_{\ell=1}^k$ satisfying $\sum_{\ell=1}^k a_{i_\ell}\leq 1$ and $\sigma_{i_1}^2<0$ are coincident. 
Let $\chi\co \Mg{0,\A'} \ra \Mg{0,\A}$ 
be the (closed immersion) replacement morphism replacing a section 
$\tau$
of weight 
$\sum_{\ell=1}^k a_{i_\ell}$ by $k$ sections $\{\sigma_{i_\ell}\}_{\ell=1}^k$ of weights 
$\{a_{i_\ell}\}_{\ell=1}^k$. Then by 
Lemma \ref{L:replacement}, the divisor $A$ pulls back to the 
sum of $\psi_\tau$ and the 
divisor $A(\sum_{\ell=1}^k a_{i_\ell}, a_{k+1},\dots, a_n)$ on
$\Mg{0,\A'}$. The latter is ample by the induction 
assumption ($\dim(\Mg{0,\A'})=\dim(\Mg{0,\A})-k+1$). 
In particular, the divisor $A+P$ 
pulls back to the sum of ample divisor and $\psi_\tau$.
By the assumption $\sigma_{i_1}^2<0$, hence 
$\psi_\tau$ has positive degree on the 
curve $B$. It follows from the projection formula that $B\cdot(A+P)\geq 0$.
\par
{\em Case of $g\geq 1$}:
Denote by $f\co \Mg{g,n}\ra \Mg{g,\A}$ the contraction morphism. 
\par 
Suppose that no sections of $(\C\ra B; \sigma_1,\dots,\sigma_n)$ 
are coincident.
Then $B$ is not entirely contained in $f(\Exc(f))$. As in the case of $g=0$, for 
$0< \epsilon\ll 1$, the divisor $(1-\epsilon)\omega+\sum a_i\sigma_i$ is nef on $\C$. 
It follows that the divisor 
\begin{align}
\label{E:perturbation}
A(\epsilon)=\left((1-\epsilon)\omega+\sum a_i\sigma_i\right)(\omega+\sum \sigma_i)=A-\epsilon(\kappa+\psi)
\end{align}
is non-negative on $B$.
It is well-known that the divisor $\kappa+\psi=12\lambda-\Delta+\psi$ is ample on $\Mg{g,n}$ \cite{CH, GKM}. 
Then for any $P\in \NS(\Mg{g,\A})\otimes\QQ$ satisfying $\vert P \vert \ll 1$, the divisor $f^*\left(P+\epsilon(\kappa+\psi)\right)$
is an effective combination of an ample divisor and an $f$-exceptional divisor. Hence, $B\cdot\left(P+\epsilon(\kappa+\psi)\right)\geq 0$.

Finally, suppose that two sections are coincident. Then $B$ lies in the image of some replacement morphism $\chi$. By Lemma \ref{L:replacement}
and the induction assumption,
the divisor $A+P$ pulls back to a sum of an ample divisor and a $\psi$ class. Since the 
 self-intersection of any section in a family of stable curves of
positive genus is non-positive, we are done. 
\end{proof}

\section{Certain log-canonical models of $\Mg{g,n}$}
By \cite[Section 3.1.1]{Hassett-weighted}, the canonical class of $\Mg{g,\A}$ is 
\[
K_{\Mg{g,\A}}=13\lambda-2\Delta_{nodal}+\psi=\lambda+\kappa-\Delta_{nodal}+\psi.
\] 
The ample divisor $A$ of Proposition \ref{P:ample-divisors} can be written as

\[
A=K_{\Mg{g,\A}}+\sum_{i<j}(a_{i}+a_{j})\Delta_{i,j}+\Delta_{nodal}-\lambda.
\] 
Since $\lambda$ is always nef, we see that
$A+\lambda=K_{\Mg{g,\A}}+\sum_{i<j}(a_{i}+a_{j})\Delta_{i,j}+\Delta_{nodal}$ is a natural ample log canonical divisor on 
$\Mg{g,\A}$.
Using this, we can give an affirmative answer to \cite[Problem 7.1]{Hassett-weighted}. The statement is cleanest for $g=0$,
when all the spaces under consideration are smooth
schemes. For $g>0$, $\Mg{g,\A}$ is only a smooth proper
Deligne-Mumford stack. We denote its coarse moduli 
space by $\overline{M}_{g,\A}$.
\begin{theorem}\label{T:hassett-spaces}
The coarse moduli space $\overline{M}_{g,\A}$ is a log canonical model of $\Mg{g,n}$. Namely,
\[
\overline{M}_{g,\A}=\proj \bigoplus_{m\geq 0}H^{0}\bigl(\Mg{g,n}, m(K_{\Mg{g,n}}+
\sum_{i<j: a_i+a_j\leq 1}(a_{i}+a_{j})\Delta_{\{i,j\}}+\Delta_{\text{\rm rest}})\bigr),
\]
where $\Delta_{\text{\rm rest}}=\Delta-\sum\limits_{i<j: a_i+a_j\leq 1}\Delta_{\{i,j\}}$.
\end{theorem}
\begin{proof}

Consider the reduction morphism $f \co \Mg{g,n} \ra \Mg{g,\A}$. Note that
\[
K_{\Mg{g,\A}}+\sum_{i<j}(a_{i}+a_{j})\Delta_{i,j}+\Delta_{n}=A+\lambda
\]
By Proposition \ref{P:ample-divisors}, the divisor $A$ is ample on $\Mg{g,\A}$. Since $\lambda$ is nef, the sum $A+\lambda$ is ample. 

It remains to observe that $f_{*}\left(K_{\Mg{g,n}}+\sum_{i<j: a_i+a_j\leq 1}(a_{i}+a_{j})\Delta_{\{i,j\}}+
\Delta_\rest\right)=A+\lambda$ and the discrepancy divisor
\[
\left(K_{\Mg{g,n}}+\sum_{i<j: a_i+a_j\leq 1}(a_{i}+a_{j})\Delta_{\{i,j\}}+\Delta_{\text{rest}}\right)-f^{*}(A+\lambda)
\]
is an effective combination of exceptional divisors with all discrepancies in the interval $[0,\infty)$. Indeed,

\begin{align*}
f^{*}(A+\lambda) &=\left(K_{\Mg{g,n}}+\sum_{i<j: a_i+a_j\leq 1}(a_{i}+a_{j})\Delta_{\{i,j\}}+\Delta_{\text{rest}}\right) \\
&+\sum_{\substack{S : |S|\geq 3\\ \sum_{i\in S} a_i \leq 1}} \left(|S|-1\right)\left(1-\sum_{i\in S} a_i\right)\Delta_S
\end{align*}
\end{proof}

\section{Nef divisors on $\Mg{0,n}$}
\label{S:F-cone}
\subsection{The case of $g=0$}
We keep the conventions of 
Sections \ref{S:main} and \ref{S:divisors}. On $\Mg{0,n}$, we let $\Delta_r$ to be 
the union of all boundary divisors $\Delta_{S}$ with $|S|=r$. 
Recall that  the canonical 
divisor of $\Mg{0,\A}$ has class
$K_{\Mg{0,\A}}=\psi-2\Delta_{nodal}$. We now restate Theorems   \ref{T:positive-divisors}  and
\ref{T:hassett-spaces} 
for 
$\Mg{0,n}$.

\begin{theorem}\label{T:hassett-spaces-0}
The Hassett's space $\Mg{0,\A}$ is a log canonical model of $\Mg{0,n}$. Namely,
\[
\Mg{0,\A}=\proj \bigoplus_{m\geq 0}H^{0}(\Mg{0,n}, m(K_{\Mg{0,n}}+\sum_{i<j}(a_{i}+a_{j})\Delta_{i,j}+\sum_{r\geq 3}
\Delta_{r})).
\]
\end{theorem}
\begin{theorem}\label{T:positive-divisors-0}
The divisor 
\[
A(a_1,\dots,a_n) =\psi+\sum_{i<j}(a_i+a_j)\Delta_{ij}-\Delta_{nodal},  
\]
is ample on $\Mg{0,\A}$ and the divisors
\begin{align*}
B(a_1,\dots,a_n) &=\sum_{i=1}^n (2a_i-a_i^2)\psi_i+\sum_{i<j}(2a_ia_j) \Delta_{ij}-\Delta_{nodal}, \\
C(a_1,\dots,a_n) &= \sum_{i=1}^n (1-a_{i})\psi_{i}+\sum_{j\neq i}(a_i+a_j)\Delta_{ij}
\end{align*}
are nef on $\Mg{0,\A}$. 
 \end{theorem}

\subsection{$\SL_2$ quotients of $(\PP^1)^n$}
In \cite{Alexeev-Swinarski}, Alexeev and Swinarski introduced a subcone, called the {\em GIT cone}, of the nef cone of $\Mg{0,n}$. The GIT cone is generated by pullbacks of natural polarizations on GIT 
quotients $(\PP^1)^n\gitq_{\vec{x}} \SL_2$. Here, we show that the GIT cone is contained in the cone generated by divisors of Theorem 
\ref{T:positive-divisors-0}. We do not preclude a possibility that two cones coincide. The idea of proof is simple: We observe that the GIT polarization on $(\PP^1)^n\gitq_{\vec{x}} \SL_2$ is proportional to $A(x_1,\dots,x_n)$. 

In what follows, we regard $(\PP^1)^n\gitq_{\vec{x}} \SL_2$ as a good moduli space of an 
Artin moduli stack\footnote{The stack is Deligne-Mumford when the linearization is typical, i.e. there are no strictly semistable 
points.} of weighted $n$\nb-pointed rational curves. 
In particular, all tautological divisors introduced in Section \ref{S:divisors}
make sense on $(\PP^1)^n\gitq_{\vec{x}} \SL_2$. 
\begin{prop}\label{P:GIT}
Let $\vec{x}=(x_1,\dots,x_n)$ be such that $x_1+\dots+x_n=2$. Then the natural GIT polarization on $(\PP^1)^n\gitq_{\vec{x}} \SL_2$
is proportional to $A(x_1,\dots,x_n)=\psi+\sum_{i<j}(x_i+x_j)\Delta_{ij}-\Delta_{nodal}$.
\end{prop}
\begin{proof}
We treat $(\PP^1)^n$ as the parameter space of $n$ ordered points on $\PP^1$. 
Consider the universal family $\pi\co (\PP^1)^n\times \PP^1\ra (\PP^1)^n$ with sections $\tau_1,\dots, \tau_n\co (\PP^1)^n\ra (\PP^1)^n\times \PP^1$. For $i=1,\dots, n+1$, set 
$H_i=\text{pr}_i^*\O_{\PP^1}(1)$. Then $\tau_i=H_i+H_{n+1}$, for $i=1,\dots, n$, and the relative dualizing sheaf of the universal 
family is $\omega=-2H_{n+1}$. We compute 
\begin{align}\label{E:git-Pn}
(\omega+\sum_{i=1}^n x_i \tau_i)(\omega+\sum_{i=1}^n \tau_i)=\left(\sum_{i=1}^n x_i H_i\right)\left((n-2)H_{n+1}+\sum_{i=1}^{n}H_i\right).
\end{align}
Pushing forward via $\pi$, we obtain that on $(\PP^1)^n$
\begin{align}\label{E:git-Pn2}
\pi_*(\omega+\sum_{i=1}^n x_i \tau_i)(\omega+\sum_{i=1}^n \tau_i)=(n-2)\sum_{i=1}^n x_i H_i.
\end{align}
By definition, an integer multiple of $\O_{(\PP^1)^n}(\sum_{i=1}^n x_i H_i)$ descends to the GIT polarization on 
$(\PP^1)^n\gitq_{\vec{x}} \SL_2$. Hence, the divisor 
class $\pi_*(\omega+\sum_{i=1}^n x_i \tau_i)(\omega+\sum_{i=1}^n \tau_i)$ descends to a multiple of the GIT polarization.
On the other hand, since $\omega$ and $\tau_i$ are defined functorially, the divisor 
class $\pi_*(\omega+\sum_{i=1}^n x_i \tau_i)(\omega+\sum_{i=1}^n \tau_i)$ descends to $A(x_1,\dots, x_n)$ on $(\PP^1)^n\gitq_{\vec{x}} \SL_2$. The statement follows.

For completeness, we record what divisor classes on $(\PP^1)^n$ descend to tautological divisor classes $\psi_i$, $\Delta_{ij}$ and 
$\Delta_{nodal}$ on $(\PP^1)^n\gitq_{\vec{x}} \SL_2$. On $(\PP^1)^n$:
$$\pi_*(-\tau^2_i)=\pi_*(-(H_i+H_{n+1})^2)=-2H_i.$$
It follows that $\O_{(\PP^1)^n}(-2H_i)$ descends to $\psi_i$. Further, since $(\tau_i-\tau_j)^2=0$ on $(\PP^1)^n\times\PP^1$, we conclude that
$$\Delta_{ij}=-\frac{1}{2}(\psi_i+\psi_j).$$
Finally, $\Delta_{nodal}=0$ because all curves parameterized by $(\PP^1)^n\gitq_{\vec{x}} \SL_2$ are $\PP^1$s. 

We conclude that the GIT polarization on $(\PP^1)^n\gitq_{\vec{x}} \SL_2$ is written in terms of tautological divisor classes
as 
$$
A(x_1,\dots, x_n)=\sum_{i=1}^n\psi_i+\sum_{i<j}(x_i+x_j)\Delta_{ij}-\Delta_{nodal}=-\frac{(n-2)}{2}\sum_{i=1}^n x_i \psi_i.
$$
\end{proof}

\subsection{Possible extensions} 

The main positivity result of Proposition \ref{P:relatively-nef} can be used to obtain
other nef divisors on $\Mg{0,n}$. We present now an example of such application.
\subsubsection{Case of $n=6$}
We consider a weight vector $\A=(1/2,1/2,1/2,1/2,1/2,1/2)$. 
Clearly, $\Mg{0,\A}\cong \Mg{0,6}$.
Since the Fulton's conjecture holds for $n=6$ (\cite[Theorem 1.2]{keel-mckernan} or
\cite[Theorem 2]{FG}), the symmetric nef cone of $\Mg{0,\A}$ has two extremal rays 
generated by divisor classes dual to $F$-curves.
In terms of tautological divisor classes they are $2\Delta_{nodal}-\psi$ and
$\psi+\Delta_s$. The latter is $\frac{3}{2}\cdot C(1/3,\dots,1/3)$, where 
$C(1/3,\dots,1/3)$ is the nef divisor
of Theorem \ref{T:positive-divisors-0}. We establish that the former is nef 
using Proposition \ref{P:relatively-nef} and a case-by-case analysis.
\begin{prop} The divisor $2\Delta_{nodal}-\psi$ is nef on $\Mg{0,\A}$.
\end{prop}
\begin{proof}
First, note that the relation $5\psi+2\Delta_s-9\Delta_{nodal}$ holds 
in the Picard group of $\Mg{0,\A}$. It follows that 
\begin{align}
\label{relation-1}
2\Delta_{nodal}-\psi=\frac{1}{9}(\psi+4\Delta_s)=\frac{2}{3}\psi+\frac{2}{3}\Delta_s-\Delta.
\end{align}
Consider a family $\C\ra B$ of $\A$-stable curves over a smooth proper 
curve $B$. 
If the total space $\C$ is not irreducible, 
then it is necessarily a union of two $(1,1/2,1/2,1/2)$-stable families over $B$. One
sees readily that $\psi=-3\Delta_{nodal}$ and $\Delta_{nodal}>0$ in this case, and so $2\Delta_{nodal}-\psi>0$.

Suppose that $\C$ is an irreducible surface. If all sections have positive self-intersection,
we are done. If there are no coincident sections, then $\Delta_s\geq 0$ and 
by Proposition \ref{P:relatively-nef} 
the divisor $\omega+\Sigma$ is nef on $\C$. It follows that 
$\psi+2\Delta_s-\Delta_{nodal}=(\omega+\Sigma)^2 \geq 0$. In particular, $\psi+4\Delta_s\geq 0$,
and we are done.

Suppose now that there are two coincident sections $\Sigma_1=\Sigma_2$, necessarily
of negative self-intersection and disjoint from $\Sigma_\rest=\sum_{i=3}^6 \Sigma_i$.
The fibered surface $\pi\co \C\ra B$ is obtained by successive blow-ups from a 
$\PP^1$-bundle 
over $B$ with a negative section $\Sigma_1$. It follows that the divisor 
$\omega+\frac{1}{3}\Sigma$ is an effective combination of the fiber class and $(-1)$-curves
in the fibers of $\pi$. Moreover, these $(-1)$-curves are disjoint from $\Sigma_1$. 
By the stability assumption, $\omega+\frac{1}{2}\Sigma_\rest$ has intersection $0$ with 
the fiber class and positive intersection with $(-1)$-curves disjoint from $\Sigma_1$. We
conclude that 
\begin{align*}
0\leq (\omega+\frac{1}{3}\Sigma)\cdot (\omega+\frac{1}{2}\Sigma_\rest) &= \frac{2}{3}\psi_1
+\frac{2}{3}\sum_{i=3}^6 \psi_i+\frac{1}{3}\sum_{3\leq i<j\leq 6}\Delta_{ij}-\Delta_{nodal}.
\end{align*}
If sections $\{\Sigma_i\}_{i=3}^6$
 are distinct, then 
 \begin{multline*}
 \frac{2}{3}\psi_1
+\frac{2}{3}\sum_{i=3}^6 \psi_i+\frac{1}{3}\sum_{3\leq i<j\leq 6}\Delta_{ij}-\Delta_{nodal} \\
\leq \frac{2}{3}\psi_1
+\frac{2}{3}\sum_{i=3}^6 \psi_i+\frac{2}{3}\sum_{3\leq i<j\leq 6}\Delta_{ij}-\Delta_{nodal}
 =\frac{2}{3}\psi+\frac{2}{3}\Delta_s-\Delta_{nodal},
 \end{multline*}
and we are done because of Equation \eqref{relation-1}.
 
 Finally, if there are two pairs of coincident sections among $\{\Sigma_i\}_{i=1}^6$, then by 
 replacing every pair of coincident sections by a section of weight $1$,
 we reduce to proving that the divisor $\frac{2}{3}\psi+\frac{2}{3}\Delta_s-\Delta_{nodal}$
 is nef on $\Mg{0, (1,1,1/2,1/2)}$. The space under consideration is isomorphic $\PP^1$.
 The degree of  $\frac{2}{3}\psi+\frac{2}{3}\Delta_s-\Delta_{nodal}$ on the
 universal family is $4/3+2/3-2=0$. 
 \end{proof}

\subsection*{Acknowledgements} 
We would like to thank David Smyth for discussions and comments regarding this work and Brendan Hassett for pointing out an error in a previous version.

\end{document}